\documentclass[a4paper,11pt]{article}

\usepackage{amsfonts,qtree,stmaryrd,textcomp}
\usepackage{pifont,amssymb,amsmath,amsthm,tabularx,graphicx}
\usepackage{tree-dvips,pictexwd,dcpic}
\usepackage{bbm}

\usepackage{stmaryrd}

\usepackage[T1]{fontenc}
\usepackage[latin1]{inputenc}
\usepackage[text={17cm,26cm},centering]{geometry}

\newtheorem{theorem}{Theorem}[section]
\newtheorem{proposition}[theorem]{Proposition}
\newtheorem{lemma}[theorem]{Lemma}
\newtheorem{corollary}[theorem]{Corollary}

\newtheorem{definition}[theorem]{Definition}

\newcommand{\dotminus}{\mathbin{{{-\mkern -9.5mu%
\mathchoice{\raise 2pt\hbox{$\cdot$}\mkern6mu}%
{\raise 2pt\hbox{$\cdot$}\mkern6mu}%
{\raise 1.5pt\hbox{$\scriptstyle\cdot$}\mkern4mu}%
{\raise 1pt\hbox{$\scriptscriptstyle\cdot$}\mkern4mu}}}}}

\newcommand{\BISH}{\mathrm{BISH}}

\newcommand{\DC}{\mathrm{DC}}

\newcommand{\LPO}{\mathrm{LPO}}

\newcommand{\Nat}{{\mathbb N}}

\newcommand{\Fi}{{\mathbb F}}

\newcommand{\DL}{\mathrm{DL}}
\newcommand{\DLs}{\DL^{*}}

\newcommand{\FNN}{\Fi(\Nat, \Nat)}
\newcommand{\FNNtwo}{\Fi(\Nat^{2}, \Nat)}
\newcommand{\FNNthree}{\Fi(\Nat^{3}, \Nat)}
\newcommand{\FXY}{\Fi(X, Y)}
\newcommand{\lex}{\mathrm{lex}}
\newcommand{\nset}{\textbf{n}}
\newcommand{\PH}{\mathrm{PH}}

\newcommand{\twoo}{2}
\newcommand{\twoset}{\textbf{\twoo}}
\newcommand{\FNtwo}{\Fi(\Nat, \twoset)}

\newcommand{\Peak}{\mathrm{Peak}}

\begin{document}

\date{}





\title{Constructive Combinatorics of Dickson's Lemma}

\author{Iosif Petrakis\\
        University of Munich\\
       petrakis@math.lmu.de}

\maketitle

\begin{abstract}
We study constructively the relations between the finite cases of Dickson's lemma.
Although there are many constructive proofs of them, the novel aspect of our
proofs is the extraction of a corresponding bound. We provide some new one-step unprovability results i.e., 
results of the form ``a finite case of Dickson's lemma does not prove in one step a stronger case of it''.
Moreover, we study the infinite cases of
Dickson's lemma from the point of view of constructive reverse mathematics. We work within Bishop's
informal system of constructive mathematics BISH. 
\end{abstract}


\section{Introduction}
\label{sec: intro}

\subsection{The finite and infinite cases of a combinatorial theorem $\tau$}
\label{subsec: comb}

According to~\cite{BM73}, p.391, the basic propositions of (classical) combinatorics 
\begin{quote}
assert, crudely speaking, that every system of a certain class possesses a large subsystem with 
a higher degree of organization than the original system.
\end{quote}

The larger the subsystem is proven to be, the stronger the corresponding theorem is. 
Suppose that $\tau$ is a theorem of combinatorics asserting for a system $S$ in a class of systems $\varSigma$ 
the existence of a subsystem $I$ of $S$ that has property $P$, which generally $S$ does not. 
In most cases property $P$ is \textit{hereditary}, i.e., if $I{'} \subseteq I$ and $P(I)$, then $P(I{'})$.
If $|X|$ denotes the cardinality of a set $X$, $l \geq 1$ and $\xi$ is a cardinal strictly larger than $\aleph_{0}$, usually
the following finite and infinite cases of $\tau$ are
considered.
\begin{enumerate}
\item The \textit{finite} case $\tau(l)$: If $|S| \geq {l}$, there is $I \subseteq S$ 
such that $|I| = l$ and $P(I)$. 
\item The \textit{strong finite} case $\tau^{*}(l)$: There is $M(l) > 0$ such that
if $l \leq |S| \leq M(l)$, there is $I \subseteq S$ 
such that $|I| = l$ and $P(I)$. 
\item The \textit{unbounded} case: If $|S| \geq {\aleph_{0}}$, then $\forall_{l \geq 1}(\tau(l))$.
\item The \textit{infinite} case $\tau(\aleph_{0})$: If $|S| \geq {\aleph_{0}}$, there is $I \subseteq S$ 
such that $|I| = \aleph_{0}$ and $P(I)$. 
\item The \textit{higher infinite} case $\tau(\xi)$:
If $|S| \geq {\xi}$, there is $I \subseteq S$ 
such that $|I| = \xi$ and $P(I)$.
\end{enumerate}

For the constructive study of such a combinatorial theorem $\tau$ a general pattern can be described.

\textbf{a.} The finite case $\tau(\l)$ is constructively proved, although there are finite combinatoric propositions,
like Friedman's Proposition B, which is provable only with the use of large cardinals
(see~\cite{Fr98} and~\cite{HHM08}), or the proposition of Paris-Harrington, which is provable 
in second-order anlysis but not in Peano arithmetic, and also lacks a constructive 
proof\footnote{On 2011, during a colloquium-talk at LMU, Veldman suggested 
to try to find such a proof.}.  

\textbf{b.} In many cases a strong case $\tau^{*}(l)$ is also
constructively proved. To find explicitly though, a bound for the strong case $\tau^{*}(l)$ is usually a 
difficult problem, and for many well-studied combinatorial theorem, like Higman's lemma, or Kruskal's 
theorem, the extraction of a bound $M(l)$ from a constructive proof of $\tau(l)$ is, to our knowledge,
not yet known.

\textbf{c.}The unbounded case $\forall_{l \geq 1}(\tau(l))$ is generally constructively proved.

\textbf{d.} The 
infinite case $\tau(\aleph_{0})$ is not constructively provable, as one usually can provide a 
Brouwerian counterexample to it, or show that $\tau(\aleph_{0})$ is constructively 
equivalent to some constructively 
unacceptable proposition, like the limited principle of omniscience LPO. It is possible though, to
find a classically equivalent formulation of $\tau(\aleph_{0})$, which admits a constructive proof
(see e.g., the intuitionistic proof of the infinite Ramsey theorem in~\cite{BV93}, or
it's constructive proof in Type Theory in~\cite{VCW12}). It is not uncommon that non-constructive proofs 
inspire, or have a constructive counterpart. E.g., minimal-bad-sequence-proofs of Higman's lemma,
or of Dickson's lemma inspired corresponding constructive (inductive) proofs of them.

\textbf{e.} The higher infinite case $\tau(\xi)$ is generally 
beyond the scope of constructive combinatorics. 

Often the proof of some case of $\tau$ is
based on the use of a \textit{repetitive argument}, that is on the  
repetition of the same proof-step for an appropriate number of times. In this
way the power of repetition of a simple, single argument is revealed.
Moreover, if a bound is extracted 
from the single proof-step, then a bound is extracted from the whole proof. 
Most of the proofs included in this paper are based on repetitive arguments.
Although from such proofs we do not extract the best possible, or optimal bounds, we find them
interesting because they are somehow ``elementary''.

\subsection{The finite and infinite cases of Dickson's lemma}
\label{subsec: versions}

Dickson's lemma is the simplest theorem of the form ``a certain quasi-order is a well-quasi-order'', and it
is connected to the the theory of Gr\"{o}bner bases and the termination of Buchberger's algorithm for finding 
them (see~\cite{CLO92} and~\cite{CP99}). This was one of the first examples of how a well-quasi-order 
can be used as a technique applied to program termination (for more on this see~\cite{VCW12}). Here we
present though, the finite and infinite cases of Dickson's lemma independently from the theory of well-quasi-orders.
First we need a definition.

\begin{definition}\label{def: good} If $X, Y$ are sets, $\FXY$ denotes the set of functions from $X$ to $Y$.
Let $k \in \Nat$ such that $k \geq 1$, 
$\alpha_{1}, \ldots, \alpha_{k} \in \FNN$,  
$(i, j) \in \Nat^{2}$ such that $i < j$, and $I \subseteq \Nat$.
The pair $(i, j)$ is called a good pair of indices for $\alpha_{1}, \ldots, \alpha_{k}$, or 
$\alpha_{1}, \ldots, \alpha_{k}$ are called good on $(i, j)$,
if $\alpha_{n}(i) \leq \alpha_{n}(j)$, for every $n \in \{1, \ldots, k\}$. 
We say that $\alpha_{1}, \ldots, \alpha_{k}$ are good on $I$, or $I$ is good for $\alpha_{1}, \ldots, \alpha_{k}$,
if $\alpha_{1}, \ldots, \alpha_{k}$ are good on every pair of indices $(i, j) \in 
I^{2}$ such that $i < j$.

\end{definition}

If $k \geq 1$ and $l \geq 2$, the following finite and infinite cases of Dickson's lemma are usually 
considered. 

\begin{enumerate}
 \item $\DL(k, l)$: If $\alpha_{1}, \ldots, \alpha_{k} \in \FNN$, there
 exists $I_{l} = \{i_{1} < i_{2} < \ldots < i_{l}\} \subset \Nat$ such that $\alpha_{1}, \ldots, \alpha_{k}$ are
 good on $I_{l}$.
 \item $\DL(k, \infty)$: If $\alpha_{1}, \ldots, \alpha_{k} \in \FNN$,
there  exists $I_{\infty} = \{i_{1} < i_{2} < \ldots < i_{n} < \i_{n+1} < \ldots\} \subseteq \Nat$
such that $\alpha_{1}, \ldots, \alpha_{k}$ are
 good on $I_{\infty}$.
 \item $\DL(k, U)$: If $\alpha_{1}, \ldots, \alpha_{k} \in \Fi(U, \Nat)$,
 where $U$ is an unbounded\footnote{That is $\forall_{n \in \Nat}\exists_{m \in \Nat}(m > n \wedge m \in U)$.} 
 subset of $\Nat$, there exists an unbounded subset $I_{U}$ of $U$ such that $\alpha_{1}, \ldots, \alpha_{k}$ are
 good on $I_{U}$.
\end{enumerate}

If $\varSigma = \mathcal{P}(\Nat)$, and $S = \nset$, where $\nset := \{0, \ldots, n-1\}$, or
$S = \Nat$, and if $P(I)$, where
$I \subseteq \nset$, for some $n \in \Nat$, or $I \subseteq \Nat$, is the hereditary property 
defined as ``the sequences $\alpha_{1}, \ldots, \alpha_{k}$
are good on $I$'', then the cases $\DL(k, l)$ and $\DL(k, \infty)$ are special cases of a combinatorial theorem
$\tau$, for which no higher infinite case is meaningful.

Note that an infinite case $\DL(\infty, 2)$ of $\DL(k, 2)$ does not hold; if we consider the sequence of 
sequences
$(\alpha_{n})_{n =1}^{\infty}$, where for every $n \geq 1$ the sequence $\alpha_{n} \in \FNN$ is
$${\alpha_{n}} = {(n, n-1, \ldots, 1, n+1, n+2, n+3, \ldots)},$$
we cannot find a pair of indices which is good for all $\alpha_{n}$;
If ${n} = {\alpha_{n}(0)}$, 
then ${\alpha_{n}(n-1)} = {1}$, for every $n \geq 1$. Hence, if $i < j$, then ${\alpha_{j+1}(i)} > {1}$, 
while ${\alpha_{j+1}(j)} = {1}$, i.e., $(i, j)$ cannot be a good pair for $\alpha_{j+1}$. This is a 
simple example of a finite combinatorial proposition the infinite case of which does not hold, even 
classically\footnote{A deeper example is related to van der Waerden's theorem. 
According to it, if $\mathbb{N}$ is partitioned into two classes, then at least one of them contains
arbitrarily long arithmetic progressions. But that does not imply that an infinite  arithmetic
progression in one of them exists (see~\cite{GRS80}, p.69).}.

The original formulation of Dickson's lemma in~\cite{Di13} is equivalent to $\DL(k, U)$, which, as we
show in section~\ref{sec: inf}, is equivalent to $\DL(k, \infty)$ and cannot be constructively accepted. 
On the other hand, the finite case $\DL(k, l)$ has already
a short constructive history. As Veldman and Bezem say in~\cite{BV93}, p.210, it was John Burgess who,
in a letter from 1983, asked for a constructive proof of $\DL(2, 2)$, which is shown to be a consequence of the 
intuitionistic Ramsey theorem in~\cite{BV93}. In~\cite{Ve04} Veldman gave an elementary
inductive, constructive proof of $\DL(k, 2)$, independently from the intuitionistic Ramsey theorem
or some special intuitionistic principle. In~\cite{CP99} Coquand and Persson gave a constructive proof
of an inductive version of $\DL(k, k)$. In~\cite{BSS01} a program is extracted from a classical 
proof of $\DL(2, 2)$, by
transforming the classical proof into a constructive one through a
refined version of $A$-translation, and the proof is
implemented in MINLOG (see also~\cite{BBS02},~\cite{Ra11},~\cite{SW11}). 
From the program extraction-point of view Dickson's lemma has been studied within
systems like Mizar, Coq and ACL2 (see\cite{Sc05},~\cite{CP99},~\cite{MM03}, respectively).
In~\cite{He04} Hertz proof-mined two classical proofs of $\DL(k, 2)$ using the Dialectica interpretation.
We refer here only to direct constructive approaches to Dickson's lemma. Since the finite cases of
Dickson's lemma follow easily from Higman's lemma, a constructive proof of the latter gives a constructive
proof of the former (see~\cite{RS93}). In~\cite{Be16} it is shown that all finite cases of 
Dickson's lemma imply Higman's lemma for
words of an alphabet with two letters.


The extraction of a bound for $\DL(k, l)$ i.e., the mining of a number
$M_{\alpha_{1}, \ldots, \alpha_{k}}(l) > 0$ out of a proof of 
$\DL(k, l)$ such that $\{i_{1} < \ldots < i_{l}\}$ is good for $\alpha_{1}, \ldots, \alpha_{k}$ and
$i_{l} \leq M_{\alpha_{1}, \ldots, \alpha_{k}}(l)$ is, surprisingly, not well-studied (neither constructively
nor classically). An exception to this is the work~\cite{BS16}, where with the use of the 
finite pigeonhole principle a strong case of $\DL(2, 2)$ is shown. It doesn't seem possible though, to
generalize this result to a method to prove strong cases of $\DL(k, 2)$, for $k > 2$. \\[1mm]
The main results of this paper are the following.

\begin{enumerate}
 \item Proposition~\ref{prp: 1l}, a strong case $\DLs(1, l)$ of $\DL(1, l)$, for every $l \geq 3$.
 \item Proposition~\ref{prp: 22}, a strong case $\DLs(2, 2)$ of $\DL(2, 2)$.
 \item Proposition~\ref{prp: 2l}, a strong case $\DLs(2, l)$ of $\DL(2, l)$, for every $l \geq 3$.
 \item We explain how our proof of Proposition~\ref{prp: 2l} generates a proof of a strong case $\DLs(k, l)$ of
 $\DL(k, l)$, where $k > 2$ and $l \geq 3$, and how the latter
 together with the proof of Proposition~\ref{prp: 22} generate
 a proof of a  strong case $\DLs(k+1, 2)$ of $\DL(k+1, 2)$.
 \item Theorem~\ref{thm-1l}, a positive formulation of the non-existence of an one step-proof of 
 $\DL(2, 2)$ from $\DL(1, l)$.
 \item Theorem~\ref{thm-2232}, a positive formulation of the non-existence of an one step-proof of 
 $\DL(3, 2)$ from $\DL(2, 2)$. 
\item Propositions~\ref{prp: DL(1, S) implies LPO} and~\ref{prp: LPO implies DL(1, S)}, which express
the constructive equivalence between $\DL(1, \infty)$ and LPO.
\end{enumerate}

Results 5 and 6 are technically the more involved and are, as far as we know, together with result 7, 
new. They are motivated by 
Corollaries~\ref{cor-1l} and~\ref{cor-2232}, respectively, which were conceived first.\\[1mm]
We work within Bishop's informal system of constructive mathematics BISH 
(see~\cite{Bi67},~\cite{BB85},~\cite{BR87}). A formal system that corresponds to BISH is CZF (see~\cite{AR10})
together with the principle of dependent choices ($\DC$), or Myhill's system CST (see~\cite{My75}).
For a recent
reconstruction of Bishop's set theory within BISH see~\cite{Pe19,Pe20,Pe21}.

\section{Strong finite cases of Dickson's lemma}
\label{sec: strong}


The strong form $\DLs(1, 2)$ of $\DL(1, 2)$, although trivial, is essential to the description of a bound
in all other strong cases $\DLs(k, l)$ of $\DL(k, l)$ presented here. 

\begin{proposition}[$\DLs(1, 2)$]\label{prp: 12}$\forall_{n \in \Nat}\forall_{\alpha \in \FNN}(\alpha(0) \leq n
\rightarrow \exists_{i < \alpha(0) + 1}(\alpha(i) \leq \alpha(i+1)))$.
\end{proposition}
\begin{proof}
If $n = 0$, then ${\alpha(0)} = {0}$, and ${i} = {0}$ is the required index. Next 
we suppose that $\forall_{\alpha \in \FNN}(\alpha(0) \leq n
\rightarrow \exists_{i < \alpha(0) + 1}(\alpha(i) \leq \alpha(i+1)))$ and we show that
$\forall_{\alpha \in \FNN}(\alpha(0) \leq n+1
\rightarrow \exists_{i < \alpha(0) + 1}(\alpha(i) \leq \alpha(i+1)))$.
Let $\alpha \in \FNN$ such that $\alpha(0) \leq n+1$. If $\alpha(0) \leq n$, we use the inductive hypothesis. If
$\alpha(0) = n+1$, then if ${\alpha(0)} \leq {\alpha(1)}$, we get ${i} = {0}$. If
${\alpha(0)} > {\alpha(1)}$, then ${\alpha(1)} \leq {n}$. By the 
inductive hypothesis on the sequence ${\alpha^{*}}$, where $\alpha^{*}(n) = \alpha(n+1)$, for every $n \in \Nat$,
there is ${j} < \alpha^{*}(0) + 1 = \alpha(1) + 1$ such that ${\alpha^{*}(j)} \leq {\alpha^{*}(j+1)}$ i.e.,
${\alpha(j+1)} \leq {\alpha(j+2)}$, and $i = j+1 < (n+1)+1 = \alpha(0) + 1$.
\end{proof}

If $\alpha \in \FNN$, we use the notation $M_{\alpha}(1, 2) := \alpha(0) + 1$ for the bound of $\DLs(1, 2)$ that
corresponds to $\alpha$. It is immediate to see that $M_{\alpha}(1, 2)$ is an optimal bound for $\DL(1, 2)$. 
The first part of the next simple corollary of $\DLs(1, 2)$ expresses that for 
each sequence $\alpha$ we can find a good pair $(i, j)$ for $\alpha$ such that $(j -i)$ is arbitrary large.
For its last part recall that the lexicographic ordering $<_{\lex}$ on $\Nat$ is defined by
$(n_{1}, m_{1}) <_{\lex} (n_{2}, m_{2}) :\leftrightarrow (n_{1} < n_{2}) \vee
(n_{1} = n_{2} \wedge m_{1} < m_{2}),$ for every $n_{1}, n_{2}, m_{1}, m_{2} \in \Nat$.

\begin{corollary}\label{crl: cor12}
 (i) For every $\alpha \in \FNN$ and $n > 0$
 $$\exists_{i \in \Nat}\left(i \leq 
 \sum_{j = 0}^{n-1}\alpha(j) \ \wedge \ \alpha(i) \leq \alpha(i+n)\right).$$
Moreover, the bound $\sum_{j = 0}^{n-1}\alpha(j)$ is the best possible i.e., 
there exists a sequence $\alpha$ such that $\alpha(i) > \alpha(i+n)$, for every
$i \leq M < \sum_{j = 0}^{n-1}\alpha(j)$.\\
(ii) If $n \in \Nat$, there is no sequence $\alpha \in \FNN$ such that $\forall_{k \in \Nat}(\alpha(k) < \alpha(k+1)
\wedge \alpha(k) < n)$.\\
(iii) There exists no function $f : \mathbb{N} \times \mathbb{N} \rightarrow \mathbb{N}$ such that 
$$f(n_{1}, m_{1}) < f(n_{2}, m_{2}) \leftrightarrow (n_{1}, m_{1}) <_{\lex} (n_{2}, m_{2}),$$
for every $n_{1}, n_{2}, m_{1}, m_{2} \in \Nat$.
\end{corollary}

\begin{proof}
 (i) If $\alpha \in \FNN$ and $n > 0$, we consider the sequence $\beta \in \FNN$ defined by
$$\beta(m)=\sum_{j<n}\alpha(m+j),$$
for every $m \in \Nat$. By $\DLs(1, 2)$ there exists $i \leq \beta(0) = \sum_{j = 0}^{n-1}\alpha(j)$ such that
\begin{align*}
 \beta(i) \leq \beta (i+1) & \leftrightarrow 
\sum_{j<n}\alpha(i+j) \leq \sum_{j<n}\alpha(i+1+j) \\
& \leftrightarrow \alpha(i) \leq \alpha(i+n).
\end{align*}
In order to show the optimality of the specified 
bound\footnote{For $n=1$ we get $\sum_{j = 0}^{n-1}\alpha(j)=\alpha(0)$, 
the optimal bound of $\DLs(1, 2)$.} consider, for an arbitrary $n>0$, any infinite sequence 
$\alpha$ extending the finite sequence $1, \underbrace{0, \ldots, 0}_{n}$. 
Clearly, $\sum_{j = 0}^{n-1}\alpha(j)=1$ and $\alpha(0) > \alpha(0+n)$, while $\alpha(1) \leq \alpha(1+n)$.\\
(ii)  Suppose that such a sequence $\alpha$ exists, and consider any infinite extension of the finite sequence
$\beta(0) = n_{0} > \beta(1) = \alpha(n_{0}) > \beta(2) = \alpha(n_{0}-1) > \ldots > \beta(n_{0}+1) = \alpha(0).$
By $\DLs(1, 2)$ there exists $i < \beta(0) +1 = n_{0} + 1$ such that 
$\beta(i) \leq \beta(i+1)$, which contradicts the supposed strict monotonicity of $\alpha$.\\
(iii)  Suppose that such a function $f$ exists. By the definition of $<_{\lex}$ we get
$(0, 0) <_{\lex} (0, 1) <_{\lex} (0, 2) <_{\lex} \ldots <_{\lex} (0, n) <_{\lex} \ldots <_{\lex} (1, 0),$
while by the supposed property of $f$ we have
$f(0, 0) < f(0, 1) < f(0, 2) < \ldots < f(0, n) < \ldots < f(1, 0),$
which is impossible by (ii). 
\end{proof}

Note that by the unbounded case $\forall_{l \geq 2}(\DL(1, l))$ we get that
$\forall_{\alpha \in \FNN}\forall_{n>0}\exists_{i,j \in \Nat}(j-i \geq n \wedge \alpha(i) \leq \alpha(j)),$
since by $\DL(1, n+1)$ there exist $i_{1} < \ldots < i_{n+1}$, such that $\alpha(i_{1}) \leq \ldots 
\leq \alpha(i_{n+1})$, therefore $i_{n+1} - i_{1} \geq n$, 
and $\alpha(i_{1}) \leq \alpha(i_{n+1})$. By Corollary~\ref{crl: cor12}(i) though,
we ``strongly'' know that the distance between the elements of the good pair is exactly $n$.

\begin{proposition}[$\DLs(1, l)$]\label{prp: 1l} If $l \geq 3$ and $\alpha \in \FNN$, there exist
$i_{1}, i_{2}, \ldots, i_{l}$, and $M_{\alpha}(1, l) \in \Nat$ such that 
$$i_{1} < i_{2} < \ldots < i_{l} \leq M_{\alpha}(1, l) \ \ \ \mbox{and}$$ 
$$\alpha(i_{1}) \leq \alpha(i_{2}) \leq \ldots \leq \alpha(i_{l}),$$
where
$$M_{\alpha}(1, l) = {\sum_{{j} = {1}}^{N}M_{j}},$$
$${N} = {\alpha(i^{(1)}) + 2},$$
$${M_{1}} = {M_{\alpha}(1, l-1)},$$
$${M_{j + 1}} = {M_{\alpha^{(j)}}(1, l-1)},$$
for every ${j} \in {\{1, ..., N-1\}}$, and $M_{\alpha}(1, l-1)$ is the bound according to $\DLs(1, l-1)$ on 
$\alpha$, $\alpha^{(j)}$ is the tail of $\alpha$ starting from the index $M_{j}$, $M_{\alpha^{(j)}}(1, l-1)$
is the bound according to $\DLs(1, l-1)$ on the sequence $\alpha^{(j)}$, and $i^{(1)}$ is the 
index determined by the application of $\DLs(1, l-1)$ on $\alpha$.

\end{proposition}

\begin{proof}Suppose first that $l = 3$. If we apply $\DLs(1, 2)$ on $\alpha$, we get an index
${i^{(1)}} \leq {\alpha(0)}$, such that ${\alpha(i^{(1)})} \leq {\alpha(i^{(1)}+1)}$. We write 
${M_{1}} = {M_{\alpha}(1, 2)} = {\alpha(0) + 1}$. If we apply $\DLs(1, 2)$ on the tail $\alpha^{(1)}$ 
of $\alpha$ starting from $M_{1}$, i.e., ${\alpha^{(1)}(n)} = {\alpha(M_{1} + n)}$, for every $n \in \Nat$, then 
we get an index ${i^{(2)}} \leq {\alpha^{(1)}(0)} = {\alpha(M_{1})}$, such that ${\alpha^{(1)}(i^{(2)})} 
\leq {\alpha^{(1)}(i^{(2)}+1)}$. We write ${M_{2}} = {\alpha(M_{1}) + 1}$. Repeating these steps ${N} =
{\alpha(i^{(1)}) + 2}$ number of times we get indices $i^{(1)} < i^{(2)} < \ldots < i^{(N)}$, such that the
application of $\DLs(1, 2)$ on $\alpha(i^{(1)}), \alpha(i^{(2)}), \ldots, \alpha(i^{(N)})$ gives the
existence of an index $i^{(k)}$, where ${k} \leq {\alpha(i^{(1)})}$, such that ${\alpha(i^{(k)})} \leq 
{\alpha(i^{(k+1)})}$. By the definition of the indices $i^{(1)} < i^{(2)} < \ldots < i^{(N)}$ we conclude 
that 
$${\alpha(i^{(k)})} \leq {\alpha(i^{(k+1)})} \leq {\alpha(i^{(k+1)}+1)}.$$  
The initial segment of $\alpha$ required to find the indices
$i^{(k)}, i^{(k+1)}$, and $i^{(k+1)}+1$ is ${M_{\alpha}(1, 3)} = 
{\sum_{{i} = {1}}^{N}M_{i}}$, where ${M_{1}} = {\alpha(0)+1}$, and 
for every ${i} \in {\{1, ..., N-1\}}$ we have that ${M_{j + 1}} = {\alpha(M_{j})+1}$.\\
If $l > 3$, we show that
$${\DLs(1, l)} \rightarrow {\DLs(1, l+1)}$$
by repeating $N$ number 
of times the application of $\DLs(1, l)$ on the corresponding tails of
$\alpha$, exactly as in the ${l} = {3}$ case. In this way we get indices 
$i_{1}^{(1)} < i_{1}^{(2)} < \ldots < i_{1}^{(N)}$, such that the application of
$\DLs(1, 2)$ on $\alpha(i_{1}^{(1)}), \alpha(i_{1}^{(2)}), ..., \alpha(i_{1}^{(N)})$ gives
the existence of an index $i_{1}^{(k)}$, such that ${\alpha(i_{1}^{(k)})} \leq {\alpha(i_{1}^{(k+1)})}$.
By the definition of the indices $i_{1}^{(1)} < i_{1}^{(2)} < \ldots < i_{1}^{(N)}$ we conclude that 
$${\alpha(i_{1}^{(k)})} \leq {\alpha(i_{1}^{(k+1)})} \leq {\alpha(i_{2}^{(k+1)})} \leq \ldots \leq
{\alpha(i_{l}^{(k+1)})}.$$ 
\end{proof}

Within the above proof the rightmost pair of the indices on which $\alpha$ weakly increases
is a pair of consecutive numbers. Generally, these indices are not consecutive. E.g., 
$$          \alpha(n) = \left\{ \begin{array}{ll}
                      0               &\mbox{, if ${n} = {2k}$}\\
                      1               &\mbox{, if ${n} = {2k+1}.$}
                     \end{array}
            \right. $$
doesn't weakly increase on any triad of consecutive numbers.

\begin{definition}\label{def: mono} Let $A$ be an inhabited set and $n \geq 1$. A \textit{coloring} of $A$
with $n$ colors,
or an \textit{$n$-coloring} of $A$, is a function $\chi : A \rightarrow \nset$.
If $a_{1}, a_{2} \in A$, the set $\{a_{1}, a_{2}\}$ is called a 
\textit{monochromatic pair} under $\chi$, if ${\chi(a_{1})} = {\chi(a_{2})}$.
A subset $B$ of $A$ is called \textit{monochromatic} under $\chi$, if every two 
elements of $B$ form a monochromatic pair. The notation $\PH(n, m, l),$ where ${n} \in {\mathbb{N}}$ 
and ${m, l} \in {{\mathbb{N}} \cup {\{\Nat\}}}$, expresses that if $\chi$ is an $n$-coloring of a
sequence of $A$ of length $m$, then this sequence contains a monochromatic subsequence $B$ of length $l$.
 
\end{definition}

Consequently, the case $\mbox{PH}(2, \Nat, l)$ of the pigeonhole principle, where 
$l \in \Nat$ and $l \geq 2$, says that if
$\chi$ is a $2$-coloring of $\{\alpha_{n} : {n} \in {\mathbb{N}}\} \subseteq {A}$, 
then $\{\alpha_{n} : {n} \in {\mathbb{N}}\}$ has a monochromatic subsequence of length $l$. 


\begin{proposition}\label{prp: stolz} ${\forall_{l \geq {2}}({\DL(1, l)} \rightarrow {\PH(2, \Nat, l)}})$.
\end{proposition}
\begin{proof}  Suppose that ${l} \geq {2}$, $\alpha \in \FNN$ and
$\chi$ is a $2$-coloring of $\{\alpha_{n} : {n} \in {\mathbb{N}}\}$. By $\DL(1, l)$ on
$\chi \circ {\alpha}: \mathbb{N} \rightarrow 2$ there are indices $i_{1} < i_{2} < ... < i_{l}$, 
such that ${\chi(\alpha_{i_{1}})} \leq {\chi(\alpha_{i_{2}})} \leq \ldots 
\leq {\chi(\alpha_{i_{l}})}$. If ${\chi(\alpha_{i_{l}})} = {0}$, therefore 
${\chi(\alpha_{i_{1}})} = {\chi(\alpha_{i_{2}})} = \ldots = {\chi(\alpha_{i_{l}})} = {0}$, 
the sequence $\alpha_{i_{1}}, \alpha_{i_{2}}, \ldots, \alpha_{i_{l}}$ is a monochromatic subsequence of $\alpha$ 
of length $l$. If ${\chi(\alpha_{i_{l}})} = {1}$, then we repeat the previous step on the tail
$\alpha_{i_{l}+1}, \alpha_{i_{l}+2}, \ldots,$ of $\alpha$. By $\DL(1, l)$, there are indices 
$n_{1} < n_{2} < ... < n_{l}$, such that
${\chi(\alpha_{i_{l}+n_{1}})} \leq {\chi(\alpha_{i_{l}+n_{2}})} \leq \dots \leq {\chi(\alpha_{i_{l}+n_{l}})}$.
If ${\chi(\alpha_{i_{l}+n_{l}})} = {0}$, then we get
a monochromatic subsequence of $\alpha$ of length $l$. If ${\chi(\alpha_{i_{l}+n_{l}})} = {1}$, 
we repeat the same procedure. It suffices to repeat the above steps at most $l$ number of 
times to find a monochromatic subsequence of $\alpha$ of length $l$.     
 \end{proof}

It is easy to provide a bound for $\PH(2, \Nat, l)$ based on the bounds determined by $\DLs(1, l)$ on
the sequences considered in the previous proof.

\begin{proposition}[$\DLs(2, 2)$]\label{prp: 22} If $\alpha, \beta \in \FNN$, there exist $i, j$ and 
$M_{\alpha, \beta}(2, 2) \in \Nat$ such that
$$i < j < M_{\alpha, \beta}(2, 2), \ \ \ \mbox{and}$$
$$\alpha(i) \leq \alpha(j) \  \wedge \ \beta(i) \leq \beta(j)),$$ 
where $${M_{\alpha, \beta}(2, 2)} = {\sum_{{j} = {1}}^{K}M_{j}},$$
$${1} \leq {K} \leq {N},$$
$${N} = {\alpha(i_{1}^{(1)}) + 2},$$
$${M_{1}} = {M_{\alpha}(1, 3)},$$ 
$${{\beta(i_{1}^{(1)})} \leq {1}} \rightarrow {{K} = {1}},$$
$${{\beta(i_{1}^{(1)})} \geq {2}} \rightarrow {{M_{2}} = {M_{\alpha^{(1)}}(1, \beta(i_{1}^{(1)})+1})},$$ 
$i_{1}^{(1)}$ is the first index of the application of $\DLs(1, 3)$ on $\alpha$ and $\alpha^{(1)}$
is the tail of $\alpha$ starting from index $M_{1}$. If ${1} \leq {j} \leq {N-1}$, 
then $${{\beta(i_{1}^{(j+1)})} \leq {\beta(i_{1}^{(j)})}-1} \rightarrow {{K} = {j+1}},$$
$${{\beta(i_{1}^{(j+1)})} \geq {\beta(i_{1}^{(j)})}} \rightarrow {{M_{j+1}} = 
{M_{\alpha^{(j)}}(1, \beta(i_{1}^{(j+1)})+1})},$$
and 
$i_{1}^{(j+1)}$ is the first index of the application of $\DLs(1, \beta(i_{1}^{(j)} +1))$ on $\alpha^{(j)}$,
where $\alpha^{(j)}$ is the tail of $\alpha$ starting from index $M_{j}$.
\end{proposition}
\begin{proof} We show that 
$$\forall_{l \geq {2}}(\DLs(1, l)) \rightarrow \DLs(2, 2),$$ 
hence by Proposition~\ref{prp: 1l} we get a proof 
of $\DLs(2, 2)$. Applying $\DLs(1, 3)$ on $\alpha$ we find indices $i_{1}^{(1)} <  
i_{2}^{(1)} < i_{3}^{(1)}$, for which ${\alpha(i_{1}^{(1)})} \leq {\alpha(i_{2}^{(1)})} \leq
{\alpha(i_{3}^{(1)})}$, based on the initial segment of $\alpha$ of length ${M_{1}} = {M_{\alpha}(1, 3)}$.
We also consider the finite sequence ${\beta(i_{1}^{(1)})}, {\beta(i_{2}^{(1)})}, {\beta(i_{3}^{(1)})}$.

Suppose that ${\beta(i_{1}^{(1)})} \leq {1}$. If we form the sequence
${\gamma(0)} = {\beta(i_{1}^{(1)})}, {\gamma(1)} = {\beta(i_{2}^{(1)})}, {\gamma(2)} = {\beta(i_{3}^{(1)})}$
and extend it in any way we like, then, by $\DLs(1, 2)$ there exists $j \leq {\gamma(0)} \leq {1}$, 
such that ${{\gamma(j)}
\leq {\gamma(j+1)}} \leftrightarrow {{\beta(i_{j+1}^{(1)})} \leq
{\beta(i_{j+2}^{(1)})}}$, while ${\alpha(i_{j+1}^{(1)})} \leq
{\alpha(i_{j+2}^{(1)})}$ also holds. Hence, in case ${\beta(i_{1}^{(1)})}
\leq {1}$, we can find a pair of indices for which $\DLs(2, 2)$ is satisfied, and then trivially ${K} = {1}$.

If ${\beta(i_{1}^{(1)})} = {\mu} \geq {2}$, we consider the tail $\alpha^{(1)}$ of $\alpha$ which 
starts from index $M_{1}$.
By $\DLs(1, {\mu}+{1})$ on $\alpha^{(1)}$ we find a finite sequence of indices $i_{1}^{(2)} <
i_{2}^{(2)} < \ldots < i_{{\mu}+{1}}^{(2)}$, for which 
$i_{1}^{(1)} <  i_{2}^{(1)} < i_{3}^{(1)} < i_{1}^{(2)} < i_{2}^{(2)} < \ldots < i_{{\mu}+{1}}^{(2)}$, such that
${\alpha^{(1)}(i_{1}^{(2)})} \leq {\alpha^{(1)}(i_{2}^{(2)})} \leq \ldots \leq
{\alpha^{(1)}(i_{{\mu}+1}^{(2)})}$. Of course, $\alpha$ also weakly increases on these indices. 
Considering $\beta(i_{1}^{(2)})$ 
we work as follows:

If ${\beta(i_{1}^{(2)})} \leq {{\mu}-1}$, then we can find the required
pair of indices using $\DLs(1, 2)$. If $\beta(i_{1}^{(2)}) \geq {\mu} = {\beta(i_{1}^{(1)})}$, we repeat
the previous step working with the tail $\alpha^{(2)}$ of $\alpha$ which starts from index
${M_{2}} = {M_{\alpha^{(1)}}(1, \beta(i_{1}^{(1)})+1)}$.

If we are at step $j$, where ${1} \leq {j} \leq {N-1}$, we find index $i_{1}^{(j+1)}$, which is 
the first index of the application of $\DLs(1, \beta(i_{1}^{(j)} +1))$ on $\alpha^{(j)}$, the tail 
of $\alpha$ starting from the index $M_{j}$.

If  ${{\beta(i_{1}^{(j+1)})} \leq {\beta(i_{1}^{(j)})}-1}$, then, by $\DLs(1, 2)$, the required pair
of indices is found, and ${{K} = {j+1}}$.

If ${{\beta(i_{1}^{(j+1)})} \geq {\beta(i_{1}^{(j)})}}$, we repeat the procedure at most 
${N} = {\alpha(i_{1}^{(1)}) + 2}$ number of times. Then indices $i_{1}^{(1)} < i_{1}^{(2)} < \ldots
< i_{1}^{(N)}$ will have been constructed for which, by the previous constructions, we have that  
$${\beta(i_{1}^{(1)})} \leq {\beta(i_{1}^{(2)})} \leq \ldots \leq
{\beta(i_{1}^{(N)})}.$$ 
Applying $\DLs(1, 2)$ on any extension of the finite sequence 
${\alpha(i_{1}^{(1)})}, {\alpha(i_{1}^{(2)})}, \ldots, {\alpha(i_{1}^{(N)})}$, we find a pair 
of indices on which $\alpha$ weakly increases. Since $\beta$ already weakly increases on them, 
we have found the required pair based on an initial segment of $\alpha, \beta$ of length at most 
${M} = {\sum_{{j} = {1}}^{K}M_{j}}$.
\end{proof}


\begin{proposition}[$\DLs(2, l), l\geq 3$]\label{prp: 2l} If $l \geq 3$ and $\alpha, \beta \in \FNN$, there exist
$i_{1}, i_{2}, \ldots, i_{l}$, and $M_{\alpha, \beta}(2, l) \in \Nat$ such that 
$$i_{1} < i_{2} < \ldots < i_{l} \leq M_{\alpha, \beta}(2, l) \ \ \ \mbox{and}$$ 
$$\alpha(i_{1}) \leq \alpha(i_{2}) \leq \ldots \leq \alpha(i_{l}),$$
$$\beta(i_{1}) \leq \beta(i_{2}) \leq \ldots \leq \beta(i_{l}),$$
where
$$M_{\alpha, \beta}(2, l) = {\sum_{{k} = {1}}^{K}M^{(k)}},$$
$$ {1} \leq {K} \leq {\Lambda},$$
$${\Lambda} = {\alpha(i^{(1)})+1},$$
and $i^{(1)}$ is the first index determined by the application of $\DLs(1, l)$ on the sequence 
${\alpha^{*}(n)} = {\alpha(i_{n})}$, where the indices $i_{n}$ are formed as follows: 
$i_{1}$ is the first component of the common good pair resulting from the application of
$\DLs(2, 2)$ on $\alpha, \beta$ requiring the initial segment of $\alpha, \beta$ of length
${M_{1}} = {M_{\alpha, \beta}(2, 2)}$, and $i_{n+1}$ is the first component of the common good pair 
resulting from the application of $\DLs(2, 2)$ on $\alpha^{(n)}, \beta^{(n)}$, which are the tails
of $\alpha, \beta$ starting from index $M_{n}$. Moreover,
$${M^{(1)}} = {\sum_{{j} = {1}}^{N_{1}}M_{j}^{(1)}},$$
$${M_{1}^{(1)}} = {M_{\alpha, \beta}(2, l)},$$
$${M_{j+1}^{(1)}} = {M_{\alpha^{(j)}, \beta^{(j)}}(2, l)},$$
while 
$${{\beta(i_{1}^{(1)})} \leq {1}} \rightarrow {{K} = {1}},$$
$${{\beta(i_{1}^{(1)})} \geq {2}} \rightarrow {{M^{(2)}} = {\sum_{{j} = {1}}^{N_{2}}M_{j}^{(2)}}},$$
where ${N_{2}} = {M_{{\alpha^{*}}^{(1)}}(1, \beta(i^{(1)})+1)}$, ${\alpha^{*}}^{(1)}$
is the tail of $\alpha^{*}$ starting from index $i_{N_{1}}$,
$${M_{1}^{(2)}} = {M_{\alpha_{2}^{(1)}, \beta_{2}^{(1)}}(2, l)}, 
{M_{j+1}^{(2)}} = {M_{\alpha_{2}^{(j)}, \beta_{2}^{(j)}}(2, l)}.$$
where $\alpha_{2}^{(1)}, \beta_{2}^{(1)}$ are the tails of $\alpha, \beta$
starting from index $M^{(1)}$ and $\alpha_{2}^{(j)}, \beta_{2}^{(j)}$ are the tails of
$\alpha, \beta$, respectively, starting from index $M_{j}^{(2)}$. 
If ${2} \leq {k} \leq {{\Lambda}-1}$, then $M^{(k+1)}$ is defined through
$M_{j}^{(k)}$'s, ${1} \leq {j} \leq {N_{k}}$, in a similar way.

\end{proposition}

\begin{proof}For simplicity we show here only the case ${l} = {3}$. Applying $\DLs(2, 2)$ on $\alpha, \beta$
using their initial segment of length ${M_{1}^{(1)}} = {M_{\alpha, \beta}(2, 2)}$ we find a common good pair of
indices $(i_{1}, j_{1})$ for them. Then we apply $\DLs(2, 2)$ on $\alpha^{(1)}, \beta^{(1)}$, 
the tails of $\alpha, \beta$ starting from index $M_{1}$, using the initial segment of them of
length ${M_{2}^{(1)}} = {M_{\alpha^{(1)}, 
\beta^{(1)}}(2, 2)}$, and we find a common good pair of indices $(i_{2}, j_{2})$ for them. We
repeat this procedure enough number of times so that the sequences ${\alpha^{*}(n)} = {\alpha(i_{n})}$,
${\beta^{*}(n)} = {\beta(i_{n})}$ reach a common good pair of indices $(i_{s}, i_{t})$ for them. 
Then $(i_{s}, i_{t}, j_{t})$ is the required good triplet for $\alpha, \beta$. In order to find the pair
$(i_{s}, i_{t})$ we need to repeat the initial procedure so many times so that for the sequences
$\alpha^{*}, \beta^{*}$ we can find a common good pair of indices. 
It is clear that $M, \Lambda$ and $i^{(1)}$ as defined above for the case $l = 3$ 
determine the bound which corresponds to this proof. 
\end{proof}

The formulation of $\DLs(3, 2)$ has a complexity similar to that of the formulation of $\DLs(2, 3)$,
while its proof 
follows the pattern of the proof of $\DLs(2, 2)$. If $\alpha, \beta, \gamma$ are given sequences,
then applying $\DLs(2, 3)$ on $\alpha, \beta$ using their initial segment of length 
${M_{1}} = {M_{\alpha, \beta}(2, 3)}$ we find indices $i_{1}^{(1)} < i_{2}^{(1)} < i_{3}^{(1)} \leq M_{1}$, 
such that both $\alpha$ and $\beta$ weakly increase on them. If ${\gamma(i_{1}^{(1)})} \leq {1}$, we are done,
while if not, we apply $\DLs(2, \mu + 1)$ on $\alpha^{(1)}, \beta^{(1)}$, the tails of $\alpha, \beta$
starting from the index $M_{1}$, where ${\mu} = {\gamma(i_{1}^{(1)})}$. Let $i_{1}^{(2)}$ be 
the first index of this application. If ${\gamma(i_{1}^{(2)})} \leq {\mu - 1}$ we stop, while if 
${\gamma(i_{1}^{(1)})} \geq {\mu}$ we repeat the procedure. At any step,
either we have found the required pair, or the sequence ${\gamma(i_{1}^{(1)})} \leq {\gamma(i_{1}^{(2)})}
\leq \ldots,$ is formed. 
Our algorithm of finding the required pair terminates with bound
${M} = {M_{\alpha, \beta, \gamma}(3, 2)}$, where $M$ is the bound within which sequences
$\alpha(i_{1}^{(1)}), \alpha(i_{1}^{(2)}), \ldots,$ and $\beta(i_{1}^{(1)}), \beta(i_{1}^{(2)}), \ldots,$
have a common good pair of indices. Consequently, this is a good pair for $\gamma$ too. 
To determine $M$ we work in a completely similar way to the determination of $M_{\alpha, \beta}(2, 3)$.

If $k > 3$, the formulations of ${\DLs(k, 2)}$, and of ${\DLs(k, l)}$, for every $l \geq {3}$, are
similar to the formulations of $\DLs(3, 2)$ and of $\DLs(3, k)$, respectively. The 
general proofs 
$${\DLs(k, 2)} \rightarrow {\DLs(k, l)},$$ 
and 
$$\forall_{l \geq {2}}(\DLs(k, l)) \rightarrow {\DLs(k+1, 2)}$$
are similar to the proofs of Propositions~\ref{prp: 2l} and the proof of $\DLs(3, 2)$, respectively.
Although we avoid here the cumbersome details of the general case, we may conclude the following
regarding our proof of $\DLs(k, l)$:
\begin{enumerate}
\item It is based on two simple repetitive arguments, a ``horizontal'' one, found in the proof 
of the implication ${\DLs(k, l)} \rightarrow {\DLs(k, l+1)}$, and a ``vertical'' one, found in the proof 
of the implication $\forall_{l \geq {2}}(\DLs(k, l)) \rightarrow {\DLs(k+1, 2)}$. Both arguments depend 
on the simplest case
$\DLs(1, 2)$, something which is not the case in other constructive proofs of the finite 
cases of Dickson's lemma (e.g., like the ones in~\cite{Ve04},~\cite{BS16}).
\item It provides a method to extract a bound ${M_{{\alpha_{1}, \ldots, \alpha_{k}}}(k, l)}$ for $\DL(k, l)$.
\item Our proof of $\forall_{l \geq {2}}(\DLs(k, l)) \rightarrow {\DLs(k+1, 2)}$ 
is the constructive analogue of the constructively non-accepted proof 
$${\DL(k, \infty)} \rightarrow {\DL(k+1, l)},$$
according to which one first applies the case $\DL(k, \infty)$ on $\alpha_{1}, \ldots, \alpha_{k}$ to determine some 
$I_{\infty} \subseteq \Nat$, which is good for $\alpha_{1}, \ldots, \alpha_{k}$, and then applies
$\DL(1, l)$ on the subsequence of $\alpha_{k+1}$ determined by $I_{\infty}$. 
Here we replaced $\DL(k, \infty)$ by $\forall_{l \geq {2}}(\DLs(k, l))$.    
\end{enumerate}

\section{One-step unprovability results}
\label{sec: un}

The results included in this section are, as far as we know new, and they are motivated by our intuition 
that it is not possible to prove
$\DL(k+1, 2)$ from a finite number of cases $\DL(k, l)$ i.e., from ``less information'' than 
$\forall_{l \geq 2}(\DL(k, l))$. First we show that no single case $\DL(1, l)$ proves $\DL(2, 2)$ ``directly 
in one step''. We give a simple example to explain what we mean: if we define 
$$(n_{1}, n_{2}) \leq (m_{1}, m_{2})) :\leftrightarrow n_{1} \leq m_{1} \ \wedge \ n_{2} \leq m_{2},$$
for every $n_{1}, n_{2}, m_{1}, m_{2} \in \Nat$, then
$$\nexists_{f \in \FNNtwo}\forall_{n_{1}, n_{2}, m_{1}, m_{2} \in \Nat}(f(n_{1}, n_{2})
\leq f(m_{1}, m_{2}) \rightarrow $$
$$(n_{1}, n_{2}) \leq (m_{1}, m_{2})),$$
since, if there was such a function $f$, then $f(0, 1) > f(1, 0) > f(0, 1)$. From this we conclude that 
$\DL(1, 2)$ doesn't prove $\DL(2, 2)$ in one step, since if there was such a function $f$ and $\alpha, \beta$ 
are given sequences, by $\DL(1, 2)$ on $(f(\alpha(n), \beta(n)))_{n}$ there are indices $i<j$ such that
$$f(\alpha(i), \beta(i)) \leq f(\alpha(j), \beta(j))$$
hence
$$(\alpha(i), \beta(i)) \leq (\alpha(j), 
\beta(j)).$$ 
A positive version of the above negation is the following, constructively stronger, formula:
$$\forall_{f \in \FNNtwo}\exists_{n_{1}, n_{2}, m_{1}, m_{2} \in \Nat}(f(n_{1}, n_{2})
\leq f(m_{1}, m_{2}) \ \wedge $$
$$(n_{1}, n_{2}) \nleq (m_{1}, m_{2})).$$
Next we prove constructively a strong form of this positive version, for arbitrary $l>1$, concluding 
that no single case $\DL(1, l)$ can prove $\DL(2, 2)$ in one step. In this way a ``meta-mathematical'' question
leads to a positive mathematical fact. First we show the following lemma.

\begin{lemma}\label{lem-one} Let $M \in \Nat$, $l > 1$ and $\alpha, \beta \in \FNN$.
$$\exists_{n_{1} < n_{2} < 
\ldots < n_{l}}(\alpha(n_{1}) = \ldots = \alpha(n_{l}) < M \ \vee $$
$$\beta(n_{1}) = \ldots = \beta(n_{l}) < M) \ \vee$$
$$\exists_{n, m \in \Nat}(M \leq \alpha(n) \leq \beta(m) \ \vee \ M \leq \beta(n) \leq \alpha(m)).$$

\end{lemma}

\begin{proof}
The number $K = (l-1)M + 1$ is the bound on 
the length of a sequence colored with the $M$ colors of $\{0, \ldots, M-1\}$
in order to have a monochromatic subsequence of 
length $l$ (this simple case of the finite pigeonhole principle has an immediate
inductive proof within $\BISH$).
If all the first $K$-terms of $\alpha$ are strictly smaller than $M$, or all the first $K$-terms of $\beta$ are
strictly smaller than $M$, 
then the conclusion follows immediately. Suppose that not all the first $K$-terms of $\alpha$ and not all 
the first $K$-terms of $\beta$ are strictly smaller than $M$. The use of the principle of the excluded middle here is 
unproblematic as the related property is decidable. Hence, there are $n_{1}, m_{1} < K$ such that 
$\alpha(n_{1}), \beta(m_{1}) \geq M$. We repeat the previous step on the tails
$\alpha^{(1)}, \beta^{(1)}$ of $\alpha, \beta$ starting from $\alpha(K+1), \beta(K+1)$, respectively.
Then again either the first $K$-terms of $\alpha^{(1)}$ are strictly smaller than $M$, or the first $K$-terms of
$\beta^{(1)}$ are strictly smaller than $M$. If not there are numbers $n_{2}, m_{2}$ such that $K < n_{2}, m_{2} < 2K$ and
$\alpha(n_{2}), \beta(m_{2}) \geq M$. We repeat this procedure at most $\Lambda = 
(\alpha(n_{1})+1)$-number of times. If the first disjunct has not been proved, applying
$\DLs(1, 2)$ on the sequence
$$\gamma(0) = \alpha(n_{1}), \ \gamma(1) = \beta(m_{1}), $$
$$\gamma(2) = \alpha(n_{2}), \ \gamma(3) 
= \beta(m_{2}), \ldots \ ,$$
we get an index $i < \Lambda$  such that
$M \leq \alpha(n_{i}) \leq \beta(m_{i}) \ \vee \ M \leq \beta(m_{i}) \leq \alpha(n_{i+1})$.
\end{proof}

It is clear that the proof also works if $M = 0$, and that $n_{1}, n_{2}, \ldots, n_{l}, n, m 
\leq B = K(\alpha(n_{1}) + 1)$ i.e., $B$ is an extracted bound. If $\phi_{1}, \ldots, \phi_{n}$ are formulas, 
then $\bigwedge_{i = 1}^{n}\phi_{i}$ ($\bigvee_{i = 1}^{n}\phi_{i}$) denotes the conjunction (disjunction)
of  $\phi_{1}, \ldots, \phi_{n}$.

\begin{theorem}\label{thm-1l} If $l > 1$ and $m \in \mathbb{N}$, then
$$\forall_{f \in \FNNtwo}\exists_{i_{1}, j_{1}, \ldots,
i_{l}, j_{l} \in \Nat}\Bigg(\bigwedge_{s = 1}^{l}(m \leq i_{s}) \wedge \bigwedge_{s = 1}^{l}(m \leq j_{s}) \ \wedge$$
$$\bigwedge_{r = 1}^{l-1} [f(i_{r}, j_{r}) \leq f(i_{r+1}, j_{r+1})] \ \wedge$$
$$\bigwedge_{1 \leq r < s \leq l} (i_{r}, j_{r}) \nleq (i_{s}, j_{s})\Bigg).$$  
 
\end{theorem}

\begin{proof}
First we show this for the cases $l = 2, 3$ and then we prove that the case $l - 2$ implies the 
case $l$, for every $l > 3$. 

If $l = 2$, then fixing $m$ and applying $\DLs(1, 2)$ on the sequence
$$\alpha(0) = f(m+1, m), \ \alpha(1) = f(m, m+1),$$
$$\alpha(2) = f(m+2, m), \ \alpha(3) = f(m, m+2), \ldots \ ,$$
i.e., 
$$\alpha(2n) = f(m + (n+1), m),$$
$$\alpha(2n+1) = f(m, m + (n+1)),$$
we get $i < \alpha(0) + 1$ such that $\alpha(i) \leq \alpha(i + 1)$. 
If $i = 2k$, for some $k \in \Nat$, then
$$f(m + k + 1, m) \leq f(m, m + k + 1),$$
and if $i = 2k + 1$, for some $k \in \Nat$, then
$$f(m, m + k + 1) \leq f(m + k + 2, m)$$
while  
$$(m + k + 1, m) \nleq (m, m + k + 1),$$
$$(m, m + k + 1) \nleq (m + k + 2, m).$$
If $l = 3$, we apply Lemma~\ref{lem-one} on
$$M = f(m+1, m+1),$$
$$l = 3, $$
$$\alpha(n) = f(m, m+n+1),$$
$$\beta(n) = f(m+n+1, m).$$
If there are $n_{1} < n_{2} < n_{3}$ such that $f(m, m+n_{3}+1) = f(m, m+n_{2}+1) = f(m, m+n_{1}+1) < M$, then 
$$(m, m+n_{3}+1) \nleq (m, m+n_{2}+1) \ \wedge$$
$$(m, m+n_{3}+1) \nleq (m, m+n_{1}+1) \ \wedge$$
$$(m, m+n_{2}+1) \nleq (m, m+n_{1}+1).$$
If there are $n_{1} < n_{2} < n_{3}$ such that
$\beta(n_{3}) = \beta(n_{2}) = \beta(n_{1}) < M$, we work similarly. 
Next we suppose that there exist indices $i, j$ such that 
$$f(m+1, m+1) \leq f(m, m+i+1) \leq f(m+j+1, m).$$
Again we conclude that
$$(m+1, m+1) \nleq (m, m+i+1) \ \wedge$$
$$(m+1, m+1) \nleq (m+j+1, m) \ \wedge$$
$$(m, m+i+1) \nleq (m+j+1, m).$$
If there exist indices $i, j$ such that $f(m+1, m+1) \leq f(m+i+1, m) \leq f(m, m+j+1)$, we work similarly. 

For the inductive step we fix $f$ and we suppose that there exist 
$i_{1}, j_{1}, i_{2}, j_{2}, \ldots, i_{l-2}, j_{l-2}$ such that
$$m+1 \leq i_{1}, j_{1}, i_{2}, j_{2}, \ldots, i_{l-2}, j_{l-2} \ \  \wedge$$
$$\bigwedge_{r = 1}^{l-3} [f(i_{r}, j_{r}) \leq f(i_{r+1}, j_{r+1})] \ \ \wedge$$
$$\bigwedge_{1 \leq r < s \leq l-2} (i_{r}, j_{r}) \nleq (i_{s}, j_{s})).$$ 
Applying Lemma~\ref{lem-one} on
$$M = f(i_{l-2}, j_{l-2}),$$
$$l,$$
$$\alpha(n) = f(m, m+n+1), $$
$$\beta(n) = f(m+n+1, m),$$
and working as in  case $l = 3$, we reach the required conclusion for $f$. Note that if $1 \leq r \leq l - 2$, 
then $(i_{r}, j_{r}) \nleq (m, m + i + 1)$ and $(i_{r}, j_{r}) \nleq (m + j + 1, m)$, since by our hypothesis
$m + 1 \leq i_{r}, j_{r}$.
\end{proof}

Next corollary is an immediate consequence of Theorem~\ref{thm-1l} 
(the condition $m \leq i_{1}, j_{1}, i_{2}, j_{2}, \ldots, i_{l}, j_{l}$ in Theorem~\ref{thm-1l}, 
which shows that many such $l$-tuples of natural numbers can be found, is not necessary to its proof). 

\begin{corollary}\label{cor-1l} If $l > 2$, then
$$\nexists_{f \in \FNNtwo}\forall_{i_{1}, j_{1}, \ldots, 
i_{l}, j_{l} \in \Nat}\Bigg(\bigwedge_{r = 1}^{l-1} [f(i_{r}, j_{r}) \leq f(i_{r+1}, j_{r+1})] \rightarrow $$
$$\bigvee_{1 \leq r < s \leq l} (i_{r}, j_{r}) \leq (i_{s}, j_{s})\Bigg).$$
 
\end{corollary}

Corollary \ref{cor-1l} can be interpreted as the mathematical formulation of the expression ``$\DL(1, l)$ doesn't prove
$\DL(2, 2)$ in one step''. If there was such a function $f$, and $\alpha, \beta$ are
given sequences, applying $\DL(1, l)$ on the sequence$(f(\alpha(n), \beta(n)))_{n}$ we would get
indices $i_{1} < \ldots < i_{l}$ such that
$$f(\alpha(i_{1}), \beta(i_{1})) \leq f(\alpha(i_{2}), \beta(i_{2})) \leq \ldots \leq 
f(\alpha(i_{l}), \beta(i_{l})).$$
Then we would have
$$\bigvee_{1 \leq r < s \leq l} (\alpha(i_{r}), \beta(i_{r})) \leq (\alpha(i_{s}), \beta(i_{s})),$$
which by the constructive interpretation of disjunction implies $\DL(2, 2)$. The inequality
$(n_{1}, n_{2}, n_{3}) \leq (m_{1}, m_{2}, m_{3})$ on $\Nat^{3}$ is defined, as in the case of 
$\Nat^{2}$, pointwisely.

\begin{theorem}\label{thm-2232}
$$\forall_{f_{1}, f_{2} \in \FNNthree}\exists_{n_{1}, n_{2}, n_{3}, m_{1}, 
m_{2}, m_{3} \in \Nat}\Bigg($$
$$f_{1}(n_{1}, n_{2}, n_{3}) \leq f_{1}(m_{1}, m_{2}, m_{3}) \ \wedge$$
$$f_{2}(n_{1}, n_{2}, n_{3}) \leq f_{2}(m_{1}, m_{2}, m_{3}) \ \wedge$$
$$(n_{1}, n_{2}, n_{3}) \nleq (m_{1}, m_{2}, m_{3})\Bigg).$$
\end{theorem}

\begin{proof} We suppose first that $f_{1}(1, 0, 0) = f_{2}(1, 0, 0) = 0$. Then 
$f_{1}(1, 0, 0) \leq f_{1}(0, m_{2}, m_{3})$, $f_{2}(1, 0, 0) \leq f_{2}(0, m_{2}, m_{3})$ and
$(1, 0, 0) \nleq (0, m_{2}, m_{3})$, for every $m_{2}, m_{3} \in \Nat$. 

Next we suppose that $f_{1}(1, 0, 0) = 0$ and $f_{2}(1, 0, 0) = l_{2} > 0$. Clearly, if there are
$m_{2}, m_{3} \in \Nat$ such that $f_{2}(0, m_{2},m_{3}) \geq l_{2}$, then $(1, 0, 0)$ and $(0, m_{2}, m_{3})$ are 
the required triplets. Taking $L = l_{2} + 1$ and $m > 0$ and applying Theorem~\ref{thm-1l} on $L, m$ and
the function
$f(i, j) = f_{1}(0, i, j)$ we find indices $i_{1}, j_{1}, \ldots, i_{L}, j_{L} \geq m$ such that
$$\bigwedge_{r = 1}^{l_{2}} [f_{1}(0, i_{r}, j_{r}) \leq f_{1}(0, i_{r+1}, j_{r+1})] \ \wedge$$
$$\bigwedge_{1 \leq r < s \leq L} (i_{r}, j_{r}) \nleq (i_{s}, j_{s})).$$ 
Next we consider the sequence $f_{2}(0, i_{1}, j_{1}), \ldots, f_{2}(0, i_{L}, j_{L})$. Either there
is a term $f_{2}(0, i_{t}, j_{t})$, where $1 \leq t \leq L$, such that $f_{2}(0, i_{t}, j_{t}) \geq l_{2}$,
which gives directly what we want to show, or all these $L$ terms are numbers strictly smaller
than $l_{2}$. But then there 
are two of them which are equal i.e., there exist $r < s$ such that
$$f_{2}(0, i_{r}, j_{r}) = f_{2}(0, i_{s}, j_{s}).$$
Clearly $(0, i_{r}, j_{r})$ and $(0, i_{s}, j_{s})$ are the required triplets. Note that both of them
are non-zero triplets, since the indices determined by Theorem~\ref{thm-1l} were larger than
$m$, and $m > 0$. 

We call the previous two cases the \textit{basic proof-step}, and the
arguments used for them work for any fixed non-zero triplet
$(k_{1}, k_{2}, k_{3})$ for which $f_{1}(k_{1}, k_{2}, k_{3}) = f_{2}(k_{1}, k_{2}, k_{3}) = 0$,
or $f_{1}(k_{1}, k_{2}, k_{3}) = 0$ and $f_{2}(k_{1}, k_{2}, k_{3}) = l_{2} > 0$. If, for example, 
$k_{2} > 0$, we consider the function $f(n, m) = f_{1}(n, 0, m)$.

Finally, we treat\footnote{Classically this case has a simpler proof. 
Given functions $f_{1}, f_{2}$ either one of them is $0$ on some non-zero triplet, or not. 
In the latter case let $\Lambda_{i} = \min\{f_{i}(n_{1}, n_{2}, n_{3} \mid (n_{1}, n_{2}, n_{3}) 
\neq (0, 0, 0)\}$ and $\Lambda = \min\{\Lambda_{1}, \Lambda_{2}\}$. If we consider the functions 
$g_{i}(n_{1}, n_{2}, n_{3}) = f_{i}(n_{1}, n_{2}, n_{3})-\Lambda$, there is a triplet on which
one of them takes the value $0$.} the case $f_{1}(1, 0, 0) = l_{1} > 0$ and $f_{2}(1, 0, 0) = l_{2} > 0$.
Without loss of generality we assume that $l_{1} \leq l_{2}$. We consider the functions
$$g_{i}(k_{1}, k_{2}, k_{3}) = f_{i}(k_{1}, k_{2}, k_{3}) \dotminus l_{1},$$
where $\dotminus$ is the modified subtraction and $i \in \{1, 2\}$. Clearly, $g_{1}(1, 0, 0) = 0$ 
and $g_{2}(1, 0, 0) = l_{2} - l_{1} \geq 0$, hence by the previous basic proof-step there exist
$$(n_{1}, n_{2}, n_{3}), (m_{1}, m_{2}, m_{3}) \neq (0, 0, 0)$$
such that
$$\bigwedge_{i = 1}^{2} (f_{i}(n_{1}, n_{2}, n_{3}) \dotminus l_{1}) \leq (f_{i}(m_{1}, m_{2}, m_{3})
\dotminus l_{1}) \ \wedge$$
$$(n_{1}, n_{2}, n_{3}) \nleq (m_{1}, m_{2}, m_{3}).$$ 
First let $f_{i}(n_{1}, n_{2}, n_{3}) \geq l_{1}$, for every $i \in \{1, 2\}$, and we consider
the following cases:
  
If $f_{i}(n_{1}, n_{2}, n_{3}) > l_{1}$, for every $i \in \{1, 2\}$, then 
$f_{i}(m_{1}, m_{2}, m_{3}) > l_{1}$, and hence 
$$\bigwedge_{i = 1}^{2} (f_{i}(n_{1}, n_{2}, n_{3}) \leq f_{i}(m_{1}, m_{2}, m_{3})) \ \wedge$$
$$(n_{1}, n_{2}, n_{3}) \nleq (m_{1}, m_{2}, m_{3}).$$ 
If $f_{1}(n_{1}, n_{2}, n_{3}) = l_{1}$ and $f_{1}(m_{1}, m_{2}, m_{3}) < l_{1}$, we repeat 
the previous basic proof-step starting from the two values $f_{2}(m_{1},
m_{2}, m_{3})$ and $f_{1}(m_{1}, m_{2}, m_{3}) < l_{1}$.
If $f_{1}(n_{1}, n_{2}, n_{3}) = l_{1}$ and $f_{1}(m_{1}, m_{2}, m_{3}) \geq l_{1}$, then 
if $f_{2}(n_{1}, n_{2}, n_{3}) > l_{1}$, then $(n_{1}, n_{2}, n_{3}), (m_{1}, m_{2}, m_{3})$ is 
the required pair of triplets, while if $f_{2}(n_{1}, n_{2}, n_{3}) = l_{1}$, we consider two cases:
If $f_{2}(m_{1}, m_{2}, m_{3}) < l_{1}$, then we repeat the basic proof-step starting from the inequality
$f_{1}(m_{1}, m_{2}, m_{3})$ and $f_{2}(m_{1}, m_{2}, m_{3}) < l_{1}$. If $f_{2}(m_{1}, m_{2}, m_{3}) 
\geq l_{1}$, then $(n_{1}, n_{2}, n_{3}), (m_{1}, m_{2}, m_{3})$ is the required pair of triplets.
If $f_{1}(n_{1}, n_{2}, n_{3}) < l_{1}$ or $f_{2}(n_{1}, n_{2}, n_{3}) < l_{1}$, we repeat 
the basic proof-step starting from $f_{1}(n_{1}, n_{2}, n_{3})$ and $f_{2}(n_{1}, n_{2}, n_{3})$.
In each case either we find the required pair of triplets, or we find a starting triplet 
on which $f_{1}$ or $f_{2}$ has less value than at the starting triplet of the previous step.
If we repeat the above steps 
at most $l_{1}$ number of times\footnote{It is easy to extract a bound from this 
proof considering the bound of Theorem~\ref{thm-1l}. Note also that the whole argument can
be rephrased as an inductive one over the minimum of the values of $f_{1}, f_{2}$ on a non-zero triplet.},
we reach a basic proof-step, where $f_{1}$ or $f_{2}$ has on some non-zero triplet the value $0$.   
\end{proof}

\begin{corollary}\label{cor-2232}
 $$\nexists_{f_{1}, f_{2} \in \FNNthree}\forall_{n_{1}, n_{2}, n_{3}, 
 m_{1}, m_{2}, m_{3} \in \Nat}\Bigg($$
 $$f_{1}(n_{1}, n_{2}, n_{3}) \leq f_{1}(m_{1}, m_{2}, m_{3}) \ \wedge$$
 $$f_{2}(n_{1}, n_{2}, n_{3}) \leq f_{2}(m_{1}, m_{2}, m_{3}) \rightarrow$$
$$(n_{1}, n_{2}, n_{3}) \leq (m_{1}, m_{2}, m_{3})\Bigg).$$
\end{corollary}

The above immediate consequence of Theorem~\ref{thm-2232} can be interpreted as a 
mathematical formulation of 
the expression ``$\DL(2, 2)$ doesn't prove $\DL(3, 2)$ in one step''. If
there were such functions $f_{1}, f_{2}$ and $\alpha_{1}, \alpha_{2}, \alpha_{3} \in \FNNthree$
are given, then applying $\DL(2, 2)$ on 
$$(f_{1}(\beta(n)))_{n}, \ \ (f_{2}(\beta(n)))_{n},$$
where, for each $n \in \mathbb{N}$, 
$$\beta(n) = (\alpha_{1}(n), \alpha_{2}(n), \alpha_{3}(n)),$$ 
we would get indices $i < j$ such that
$$f_{1}(\alpha_{1}(i), \alpha_{2}(i), \alpha_{3}(i)) \leq f_{1}(\alpha_{1}(j), 
\alpha_{2}(j), \alpha_{3}(j)) \ \wedge $$
$$f_{2}(\alpha_{1}(i), \alpha_{2}(i), 
\alpha_{3}(i)) \leq f_{2}(\alpha_{1}(j), \alpha_{2}(j), \alpha_{3}(j))$$
which would imply
$$(\alpha_{1}(i), \alpha_{2}(i), \alpha_{3}(i)) \leq (\alpha_{1}(j), \alpha_{2}(j), \alpha_{3}(j)).$$

\section{On the infinite cases of Dickson's lemma}
\label{sec: inf}

In this section we study the infinite cases of Dickson's lemma from the point of view of
constructive reverse mathematics (for more information on this subject see~\cite{Is06}).
First we show the equivalence between the various infinite cases of Dickson's lemma.


\begin{proposition}\label{prp: infinity} If $k > 1$, the following are equivalent.\\
(i) $\DL(1, \infty)$.\\
(ii) $\DL(k, \infty)$.\\
(iii) $\DL(1, U)$.\\
(iv) $\DL(k, U)$.
\end{proposition}

\begin{proof} (i) $\rightarrow$ (ii)
DL$(1, \infty)$ is the first step in the inductive proof of  DL$(k, \infty)$. It is also used 
in the proof of the inductive step ${DL(k, \infty)} \leftrightarrow {DL(k+1, \infty)}$. If 
$\alpha_{1}, \alpha_{2}, \ldots, \alpha_{k+1} \in \FNN$, by 
$\DL(k, \infty)$, there is a sequence $i_{1} < i_{2} < i_{3} < \ldots$, such that ${\alpha_{m}(i_{1})} 
\leq  {\alpha_{m}(i_{2})} \leq {\alpha_{m}(i_{3})} \leq \ldots$, for evey ${m} \in {\{1, 2, \ldots, k\}}$. 
If we apply $\DL(1, \infty)$ on the sequence $\alpha_{m+1}(i_{1}), \alpha_{m+1}(i_{2}), 
\alpha_{m+1}(i_{3}), \ldots$,
we get a weakly increasing subsequence of it. By hypothesis, the sequences $\alpha_{1}, \alpha_{2}, \ldots,
\alpha_{k}$ weakly increase on its indices too.\\
The implication (ii) $\rightarrow$ (i) is trivial.\\
Next we show that (i) $\rightarrow$ (iii). With the use of the principle of dependent choices $\DC$ a sequence 
$s_{0} < s_{1} < ... < s_{n} < s_{n+1} < \ldots$, of elements of $U$ is constructed.
By $\DL(1, \infty)$ on the sequence $\alpha^{*}$, where
${\alpha^{*}(n)} = {\alpha(s_{n})}$, for every $n \in \Nat$, a subsequence $(k(n))_{n \in \Nat}$ is formed 
on which $\alpha$ is 
good. But then $\alpha$ is also good on ${M} = {\{s_{k(n)} : {n} \in {\mathbb{N}}\}}$, and 
${M}$ is an unbounded subset of ${\mathbb{N}}$.\\
The equivalence (iii) $\leftrightarrow$ (iv) is shown as the equivalence (i) $\leftrightarrow$ (ii).\\
Finally we show that (iii) $\rightarrow$ (i). If we take ${U} = {\mathbb{N}}$, then by $\DL(1, U)$
there exists ${M}$ unbounded subset of ${\mathbb{N}}$ such that ${i < j} \rightarrow
{{\alpha(i)} \leq {\alpha(j)}}$, for every ${i, j} \in M$. With the use of $\DC$
a sequence $m_{0} < m_{1} < \ldots < m_{n} < m_{n+1} < \ldots,$
is formed in $M$ such that $\alpha(m_{0}) \leq \alpha(m_{1}) \leq \ldots \leq \alpha(m_{n}) \leq \alpha(m_{n+1}) 
\leq \ldots$ .
\end{proof}

In contrast to $\forall_{l \geq {2}}(\DL(1, l))$, the infinite case $\DL(1, \infty)$ is not 
constructively acceptable. In~\cite{Ve04} Veldman gave a Brouwerian counterexample to 
$\DL(1, \infty)$. Here we show its constructive equivalence to LPO, which is the following formula
$$\forall_{\alpha \in \FNtwo}\bigg(\exists_{n \in \Nat}(\alpha(n) = 1) \vee 
\forall_{n \in \Nat}(\alpha(n) = 0)\bigg).$$
LPO is only classically true and a taboo for all varieties of
constructive mathematics. Next we show that DL$(1, \infty)$ implies LPO.

\begin{proposition}\label{prp: DL(1, S) implies LPO} ${{\DL(1, \infty)} \rightarrow {\LPO}}$.
\end{proposition}

\begin{proof} We prove that if ${\alpha} \in \FNtwo$, then
$\exists_{n \in \Nat}(\alpha(n) = 1) \vee 
\forall_{n \in \Nat}(\alpha(n) = 0)$, which is trivially equivalent to the original formulation of LPO.
Applying $\DL(1, \infty)$ on $\alpha$ we get a sequence of indices 
$i_{1} < i_{2} < i_{3} < \ldots$,
such that ${\alpha(i_{1})} \leq {\alpha(i_{2})} \leq {\alpha(i_{3})} \leq \ldots$ \ . Note that if 
${\alpha(i_{1})} = {1}$, then
${\alpha(i_{n})} = {1}$, for each ${n} \geq 1$. Through $\alpha$ we define a sequence 
$\beta \in \FNtwo$ by 
$$\beta(n) = \left\{ \begin{array}{ll}
              1                        &\mbox{, if $\forall_{m \leq {i_{n}}}({\alpha(m)} = {1})$}\\
              0                        &\mbox{, if $\exists_{m \leq {i_{n}}}({\alpha(m)} = {0})$.}
                     \end{array}
            \right. $$
By DL$(1, \infty)$ on $\beta$, a sequence of indices $j_{1} < j_{2} < j_{3} < \ldots,$ is formed
such that ${\beta(j_{1})} \leq {\beta(j_{2})} \leq {\beta(j_{3})} \leq \ldots$ \ . If ${\beta(j_{1})} = {0}$,
then 
$\exists_{m \leq {i_{j_{1}}}}({\alpha(m)} = {0})$, and the conclusion of LPO is reached. 
If ${\beta(j_{1})} = {1}$, then again ${\beta(j_{m})} = {1}$, for each ${m} \in {\mathbb{N}}$. 
In that case we show that $\forall_{n \in \Nat}({\alpha(n)} = {1})$. Consider a fixed ${n} \in {\mathbb{N}}$.
Then we can find $i_{k} > n$ and $j_{l} > k$. Since ${\beta(j_{l})} = {1}$, $\forall_{m \leq 
{i_{j_{l}}}}({\alpha(m)} = {1})$. But $k < j_{l}$ implies that $n < i_{k} < i_{j_{l}}$, therefore 
${\alpha(n)} = {1}$.
\end{proof}

In~\cite{Ra11}, p.148, Ratiu asked whether $\DL(k, U)$ implies LPO. By 
Propositions~\ref{prp: infinity}
and \ref{prp: DL(1, S) implies LPO} we get an affirmative answer to this.
%

\begin{proposition}\label{prp: LPO implies PEM} If $P(n)$ is a decidable predicate on
$\mathbb{N}$, then $\LPO \rightarrow [\forall_{n \in \Nat}(P(n)) \vee \exists_{n \in \Nat}(\neg{P(n)})]$.
\end{proposition}
\begin{proof}If we define
$$  \alpha(n) = \left\{ \begin{array}{ll}
              1                        &\mbox{, if $\neg{P(n)}$}\\
              0                        &\mbox{, if $P(n)$}
                     \end{array}
            \right. $$
then LPO on $\alpha$ is exactly $\forall_{n \in \Nat}(P(n)) \vee \exists_{n \in \Nat}(\neg{P(n)})$.
\end{proof}

\begin{definition}
If ${i} \in {\mathbb{N}}$ and $\alpha \in \FNN$, we call $i$ a peak for 
$\alpha$, $\Peak_{\alpha}(i)$, if and only if
$\forall_{n > i}(\alpha(i) > \alpha(n))$. 
\end{definition}

\begin{proposition}\label{prp: LPO implies DL(1, S)}If $\alpha \in \FNN$,
then $$\LPO \rightarrow \forall_{i \in \Nat}\bigg(\Peak_{\alpha}(i) \vee 
\exists_{n > i}(\alpha(i) \leq \alpha(n))\bigg).$$
\end{proposition}
\begin{proof} If ${\mathbb{N}_{> {i}}} = {\{{n} \in \mathbb{N} : {n} > {i}\}}$ and
$e: \mathbb{N} \rightarrow \mathbb{N}_{> {i}}$ is the bijection defined by 
${e(n)} = {(n+1) + i}$, for every $n \in \Nat$,
then for the decidable predicate $${P_{i}(n)} \leftrightarrow {\alpha(e(n)) < \alpha(i)}
\leftrightarrow {\alpha((n+1)+i) < \alpha(i)},$$ 
Proposition~\ref{prp: LPO implies PEM} gives
$$\forall_{n \in \Nat}(\alpha((n+1)+i) < \alpha(i)) \vee \exists_{n \in \Nat}(\alpha((n+1)+i) \geq \alpha(i)).$$ 
Therefore, either $i$ is a peak for $\alpha$, or there is an index after $i$ of at least the same 
value as $i$ under $\alpha$, which is exactly what we need to prove. 
\end{proof}

\begin{proposition}\label{prp: LPO implies DL(1, S)} ${{\LPO} \rightarrow {\DL(1, \infty)}}$.
\end{proposition}
\begin{proof} Through the previous decidability of $\Peak_{\alpha}(i)$ we define a sequence 
${\beta} \in \FNtwo$ by 
$$  \beta(n) = \left\{ \begin{array}{ll}
              0                        &\mbox{, if $\exists_{m > n}({\alpha(n)} \leq {\alpha(m)})$}\\
              1                        &\mbox{, if $\Peak_{\alpha}(n)$}
                     \end{array}
            \right. $$
By LPO, if $\forall_{n \in \Nat}({\beta(n)} = {0}) \leftrightarrow 
\forall_{n \in \Nat}\exists_{m > n}({\alpha(n)} \leq {\alpha(m)})$, then, since $0$ is positively 
not a peak for $\alpha$, $\exists_{n_{1} > 0}(\alpha(0) 
\leq {\alpha(n_{1})})$. Similarly, $\exists_{n_{2} > n_{1}}({\alpha(n_{1})} \leq {\alpha(n_{2})})$,
and so on. By $\DC$ a sequence ${0} = {n_{0}} < n_{1} <
n_{2} < \ldots$, is constructed such that ${\alpha(n_{0})} \leq {\alpha(n_{1})} \leq {\alpha(n_{2})}
\leq \ldots$ \ .
If $\exists_{n \in \Nat}(\beta(n) = 1) \leftrightarrow \exists_{n \in \Nat}(\Peak_{\alpha}(n))$,
and if we consider the
tail of $\alpha$ 
$$\alpha(n+1), \alpha(n+2), \alpha(n+3), \ldots,$$    
then 
$${\alpha(j)} \in {\{0, 1, \ldots, \alpha(n)-1\}},$$ 
for every ${j} \geq {n+1}$. Since this tail of
$\alpha$ is a new sequence, then either it has positively no picks, and the previous case is applied,
or there is some index $n+m+1$ which is a peak for the sequence $\alpha(n+1), \alpha(n+2), \alpha(n+3), \ldots$ \ . 
Since ${\alpha(n+m+1)} \in {\{0, 1, \ldots, \alpha(n)-1\}}$, then 
${\alpha(j)} \in {\{0, 1, \ldots, \alpha(n)-2\}}$, 
for every ${j} > {n+m+1}$. After at most $\alpha(n)$-$1$ number of steps we will have found a tail of 
$\alpha$ with no peaks. If we apply then the argument of the first case, we reach our conclusion. 
\end{proof}

%

In analogy to Proposition~\ref{prp: stolz} we show that $\DL(1, \infty)$ implies Stolzenberg's principle
$\PH(2, \Nat, \Nat)$.


\begin{proposition}\label{prp: DL implies SP} ${{\DL(1, \infty)} \rightarrow {\PH(2, \Nat, \Nat)}}$.
\end{proposition}
\begin{proof} Suppose that $\alpha \in \FNN$ and that $\chi$ is 
a $2$-coloring of $\{\alpha_{n} : {n} \in {\mathbb{N}}\}$. By $\DL(1, \infty)$ on $\chi 
\circ {\alpha}: \mathbb{N} \rightarrow 2$ there are indices 
$i_{1} < i_{1} < i_{3} < \ldots$, such that ${\chi(\alpha_{i_{1}})} \leq 
{\chi(\alpha_{i_{2}})} \leq {\chi(\alpha_{i_{3}})} \leq \ldots$ \ . 
Since ${{\DL(1, \infty)} \rightarrow {\LPO}}$, either all terms 
of $[\chi(\alpha_{i_{n}})]_{n}$ are $0$, or there is a term $\alpha_{i_{n}}$ such that 
${\chi(\alpha_{i_{n}})} = {1}$. In the first case $(\alpha_{i_{n}})_{n}$ itself is monochromatic,
while in the second the tail $\alpha_{n}, \alpha_{n+1}, \alpha_{n+2} \ldots$, of $\alpha$ is monochromatic.
\end{proof}

\section{Concluding remarks}
\label{sec: concl}

The extraction of a bound $M_{\alpha_{1}, \ldots, \alpha_{k}}(l)$ from our proof of $\DL(k, l)$ resembles 
the extraction of a term out of a proof in the field of program extraction. It is an
example of term extracted in an informal system of mathematics, like BISH.

The following open questions, or tasks need to be addressed in future work. 

\begin{enumerate}
\item To study further these terms
$M_{\alpha_{1}, \ldots, \alpha_{k}}(l)$, since by Berger's constructive proof in~\cite{Be16} of 
Higman's lemma for words of an alphabet with two letters by the finite cases of Dickson's lemma, 
a bound for this case of Higman's lemma can be formulated.
\item Results like Proposition~\ref{prp: 1l} have already been implemented in MINLOG. The implementation forced 
the inductive formulation of appropriate lemmas that cover the repetitive arguments used in the informal proofs.
It will be interesting to codify formally the more complex repetitive arguments found in the 
rest constructive proofs presented here. 
\item To extend the tools found in the proofs of Theorems~\ref{thm-1l} and~\ref{thm-2232} in 
order to prove these results in complete generality.
\item To extend our study of the finite and infinite cases of Dickson's lemma to a similar 
study of the finite and infinite cases of combinatorial theorems like
Higman's lemma, or Kruskal's theorem.

\end{enumerate}

\end{document}